\documentclass[12pt]{article}
\usepackage{ustars}

\usepackage{amsthm,amsmath,amssymb}

\usepackage{graphicx}

\usepackage[colorlinks=true,citecolor=black,linkcolor=black,urlcolor=blue]{hyperref}

		\usepackage{float}

\def\ZZ{{\mathbb Z}}
\def\RR{{\mathbb R}}

\theoremstyle{plain}
\newtheorem{theorem}{Theorem}
\newtheorem{lemma}[theorem]{Lemma}
\newtheorem{corollary}[theorem]{Corollary}
\newtheorem{proposition}[theorem]{Proposition}

\theoremstyle{definition}

\theoremstyle{remark}

\title{\bf New Upper Bounds on the Distance Domination Numbers of Grids}

\author{Armando Grez \\
\small Department of Mathematics \\[-0.8ex]
\small Florida Gulf Coast University \\[-0.8ex] 
\small Fort Myers, Florida, 33965\\
\small\tt agrez@eagle.fgcu.edu\\
\and
 Michael Farina\qquad \\
\small Department of Mathematics\\[-0.8ex]
\small Florida Gulf Coast University \\[-0.8ex]
\small Fort Myers, Florida, 33965\\
\small\tt mafarina6867@eagle.fgcu.edu
}

\date{\dateline{July 15, 2014}{ Sept 1, 2014}{Oct 30, 2014}\\
\small Mathematics Subject Classifications: 05C69, 05C12, 05C30}

\begin{document}

\maketitle


\begin{abstract}
In his 1992 Ph.D. thesis Chang identified an efficient way to dominate $m \times n$ grid graphs and conjectured that his construction gives the 
most efficient dominating sets for relatively large grids.
In 2011 Gon\c{c}alves, Pinlou, Rao, and Thomass\'e proved Chang's conjecture, establishing a closed formula for the domination number of a grid.
In March 2013  Fata, Smith and Sundaram 
established upper bounds for the $k$-distance domination numbers of grid graphs by generalizing Chang's construction of dominating sets to $k$-distance dominating sets. 
In this paper we improve the upper bounds established by Fata, Smith, and Sundaram for the $k$-distance domination numbers of grids.  
\end{abstract}

\section{Introduction}
Let $G = (V,E)$ denote a graph with vertex set $V$ and edge set $E$.  We say that a subset $S$ of $V$ is a \textit{dominating set of $G$} 
if every vertex in $G$ is either in $S$ or adjacent to at least one vertex in $S$. 
The \textit{domination number} of a graph $G$ is defined to be the cardinality of the smallest dominating set in $G$ and is denoted by $\gamma{(G)}$.

We define the \emph{distance} between two vertices $v,w \in V$ to be the minimum number of edges in any path connecting $v$ and $w$ in $G$.  
We denote the distance between $v$ and $w$ by $d(v,w)$.  
We say that a set $S$ is a $k$-\emph{distance dominating set} of $G$ if every vertex $v$ in $G$ is either in $S$ or there is a vertex $w \in S$ with $d(v,w) \leq k$,
and we define the $k$-\emph{distance domination number} of $G$ to be the size of the smallest $k$-distance dominating set of $G$.
For a comprehensive study of graph domination and its variants we refer the interested reader to the two excellent texts by Haynes, Hedetniemi and Slater \cite{HayHedSla98,HayHedSla98b}.

This paper studies $k$-distance domination numbers on $m \times n$ grid graphs, which generalize domination numbers of grid graphs.    
For the past three decades, mathematicians and computer scientists
searched for closed formulas to describe the domination numbers of $m \times n$ grids.  This search was recently rewarded with a proof of a 
closed formula for the domination number of any $m \times n$ grid with $m \geq n \geq 16$ \cite{GonPinRaoTho11}.  
We recount a brief history of the investigation here, and henceforth we let $G_{m,n}$ denote an $m \times n$ grid graph.

In 1984, Jacobson and Kinch \cite{JacKin84} started the hunt for domination numbers of grids by publishing closed formulas for the values of $\gamma(G_{2,n})$, $\gamma(G_{3,n}),$ and $\gamma(G_{4,n})$. 
In 1993, Chang, Clark, and Hare \cite{ChaCla93} extended these results by finding formulas for $\gamma(G_{5,n})$ and $\gamma(G_{6,n})$. 
In his Ph.D. thesis, Chang \cite{Cha92} constructed efficient dominating sets for $G_{m,n}$ proving that when $m$ and $n$ are greater than $8$, 
the domination number $\gamma(G_{m,n})$ is bounded above by the formula
\begin{equation} \gamma(G_{m,n}) \leq \left \lfloor \frac{(n+2)(m+2)}{ 5}\right \rfloor - 4 . \end{equation} \label{chang} Chang also conjectured that 
equality holds in Equation (\ref{chang}) when $n \geq m \geq 16$.  
 In an effort to confirm Chang's conjecture, a number of mathematicians and computer scientists began exhaustively computing the values of $\gamma(G_{m,n})$.
In 1995, Hare, Hare, and  Hedetniemi \cite{HarHarHed86} developed a polynomial time algorithm to compute $\gamma(G_{m,n} )$ when $m$ is fixed.    
Alanko, Crevals, Isopoussu, \"Ostergard, and Petterson \cite{Ala11} computed $\gamma(G_{m,n})$ for $m,n \leq 29$ in addition to $m \leq 27$ and $n \leq 1000$.
Finally in 2011, Gon\c{c}alves, Pinlou, Rao, and Thomass\'{e} \cite{GonPinRaoTho11} confirmed Chang's conjecture for all $n \geq 16$.  
Their proof uses a combination of analytic and computer aided techniques for the large cases $(n \geq m \geq 24)$ and exhaustive calculations for all smaller cases.
 
While the concept of graph domination has been generalized in countless ways including distance domination, $R$-domination, double-domination, and $(t,r)$-broadcast domination to name just a few 
\cite{Sla76, Hen98, HarHay00, JotPusSugSwa2011,BIJM2014}, relatively little is known about these other domination theories in grid graphs.  However, in 2013, Fata, Smith, and Sundaram 
generalized Chang's construction of dominating sets for grids to construct distance dominating sets that 
give the following upper bound on $k$-distance domination numbers of grids \[ \gamma_k(G_{m,n}) \leq \left \lceil \frac{(m+2k)(n+2k)}{2k^2+2k+1} + \frac{2k^2+2k+1}{4}  \right \rceil \] \cite[Theorem V.10]{FatSmiSun13}.    
In 2014, Blessing, Insko, Johnson, and Mauretour improved the previous upper bounds on 2-distance domination number to 
$ \gamma_2\left (G_{m,n} \right ) \leq \left \lfloor \frac{(m+4)(n+4)}{13} \right \rfloor - 4 $ for large $m$ and $n$, but they did not consider
$\gamma_k \left ( G_{m,n} \right )$ for $k \geq 3$  \cite[Theorem 3.7]{BIJM2014}.

The main result of this paper improves the upper bounds established by Fata, Smith, and Sundaram:
\begin{theorem} 
Assume that $m$ and $n$ are greater than $2(2k^2+2k+1)$.  Then the $k$-distance domination number of an $m \times n$ grid graph $G_{m,n}$ is bounded above by the following formula: 
\[ \gamma_{k} (G_{m,n} ) \leq  \left \lfloor  \frac{ (m+2k)(n+2k) }{2k^2+2k+1} \right \rfloor -4  .\]
\end{theorem}

Figure \ref{tab:my_label} illustrates how our main theorem improves on the bounds for $3$-distance domination number $\gamma_3(G_{m,n})$ given by Fata, Smith, and Sundaram in 2013.

\begin{table}[ht]
\centering
\begin{tabular}{c c c c}
 \hline \hline
  M & N & New Bound & Old Bound\\
 \hline
 51 & 52 & 128 & 139\\
 53 & 54 & 137 & 148\\
 55 & 56 & 147 & 158\\
 57 & 58 & 157 & 168\\
 59 & 60 & 167 & 178\\
 61 & 62 & 178 & 189\\
 63 & 64 & 189 & 200\\
 65 & 66 & 200 & 211\\

 \hline
 
\end{tabular}

\caption{Comparing upper bounds for $\gamma_3(G_{m,n}) $ }
\label{tab:my_label}
\end{table}

The rest of this paper proceeds as follows.
In Section \ref{section2} we describe an embedding of $G_{m,n}$ into the integer lattice $\ZZ^2$ and the $k$-distance neighborhood $Y_{m+2k,n+2k}$ of $G_{m,n}$.
Then we describe a family of efficient dominating sets for $\ZZ^2$ as the inverse images of a ring homomorphism 
$\phi_k: \ZZ^2 \rightarrow \ZZ_{2k^2+2k+1}$.
In Section \ref{section3} we prove that there exists an $\bar{\ell} \in \ZZ_{2k^2+2k+1}$ such 
that $\left |\phi_k^{-1} \left ( \bar{\ell} \right )  \cap Y_{m+2k,n+2k}  \right | \leq \left \lfloor  \frac{ (m+2k)(n+2k) }{2k^2+2k+1} \right \rfloor$
in Corollary \ref{boundY}.  In Section \ref{section4} we prove that when $m$ and $n$ are sufficiently large, we can remove at least one vertex from each 
corner of $\phi_k^{-1}  \left ( \bar{\ell} \right )  \cap Y_{m+2k,n+2k} $ to obtain a dominating set for $G_{m,n}$ in Lemma \ref{remove}.  
Our main result then follows immediately from Corollary \ref{boundY} and Lemma \ref{remove}.

\section{$k$-Distance Dominating Sets in $\ZZ^2$} \label{section2}
Let $\ZZ \times \ZZ = \ZZ^2$ denote the integer lattice in $\RR^2$. We embed an $m \times n$ grid graph $G_{m,n}$  into $\ZZ^2$ by identifying $G_{m,n}$ with the following subset of $\ZZ^2$
\[ G_{m,n} = \{ (i,j) \in \ZZ^2 : 0 \leq i \leq m-1 \text{ and } 0 \leq j \leq n-1 \}. \]
We define a neighborhood $Y_{m+2k, n+2k}$ around $G_{m,n}$ in $\ZZ^2$ by adding $k$ rows and columns to the boundary of $G_{m,n}$.  
That is \[ Y_{m+2k, n+2k} = \{ (i,j) \in \ZZ^2: -k \leq i \leq m+k-1 \text{ and } -k \leq j \leq n+k-1 \}. \]

Fata, Smith, and Sundaram noted that a $k$-distance neighborhood of a vertex in $\ZZ^2$ is a diamond-shaped collection of vertices containing at most $2k^2+2k+1$ elements \cite[Lemma V.3]{FatSmiSun13}. To condense our notation, we will denote the number of vertices in a $k$-distance neighborhood
by $p = 2k^2+2k+1$.
We will now describe a family of dominating sets of the lattice $\ZZ^2$ as the inverse images under a ring homomorphism.  
We define a homomorphism $\phi_k: \ZZ \times \ZZ \rightarrow \ZZ_p$ by $(i,j) \mapsto (k+1) i+kj$. 
Let $\bar{\ell}$ denote an element of $\ZZ_p$.  One can easily verify that $\phi_k^{-1}\left ( \bar{\ell}  \right )$ is a 
$k$-distance dominating set of $\ZZ^2$ \cite[Lemma V.8]{FatSmiSun13}.  The inverse image $\phi_2^{-1}\left(\bar{0}\right)$ and the $2$-distance neighborhoods of a few of its elements are depicted in Figure \ref{inverseimage}.  
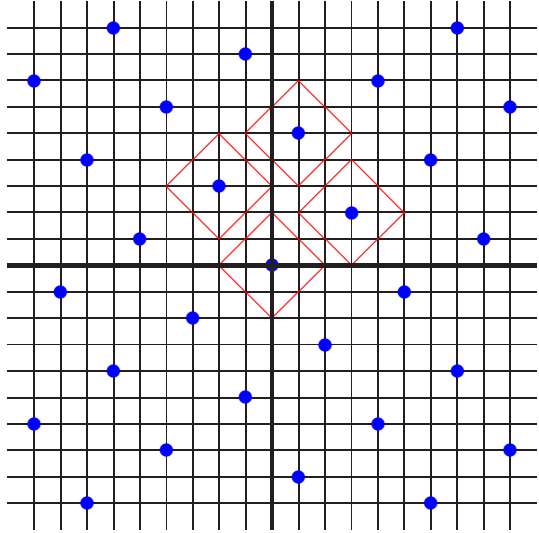
\begin{figure}[h]
\begin{picture}(210,210)(0,0)
	\linethickness{.1 mm}
	\multiput(110,10)(0,10){19}{\line(1,0){200}}
	\multiput(120,0)(10,0){19}{\line(0,1){200}}
        \put(210,100){\color{blue}{\circle*{5}}}
        \put(210,80){\color{red}{\line(1,1){20}}}   
        \put(210,80){\color{red}{\line(-1,1){20}}} 
        \put(210,120){\color{red}{\line(1,-1){20}}}   
        \put(210,120){\color{red}{\line(-1,-1){20}}}   
        
        \put(240,120){\color{blue}{\circle*{5}}}
        \put(240,100){\color{red}{\line(1,1){20}}}   
        \put(240,100){\color{red}{\line(-1,1){20}}} 
        \put(240,140){\color{red}{\line(1,-1){20}}}   
        \put(240,140){\color{red}{\line(-1,-1){20}}}

        \put(180,80){\color{blue}{\circle*{5}}}

        \put(150,60){\color{blue}{\circle*{5}}}
        \put(120,40){\color{blue}{\circle*{5}}}
        \put(270,140){\color{blue}{\circle*{5}}}
        \put(300,160){\color{blue}{\circle*{5}}}

        \put(190,130){\color{blue}{\circle*{5}}}
        
        \put(190,110){\color{red}{\line(1,1){20}}}   
        \put(190,110){\color{red}{\line(-1,1){20}}} 
        \put(190,150){\color{red}{\line(1,-1){20}}}   
        \put(190,150){\color{red}{\line(-1,-1){20}}}   
        
        \put(220,150){\color{blue}{\circle*{5}}}
        
        \put(220,130){\color{red}{\line(1,1){20}}}   
        \put(220,130){\color{red}{\line(-1,1){20}}} 
        \put(220,170){\color{red}{\line(1,-1){20}}}   
        \put(220,170){\color{red}{\line(-1,-1){20}}}   
        
        \put(160,110){\color{blue}{\circle*{5}}}
        
        \put(130,90){\color{blue}{\circle*{5}}}
        \put(250,170){\color{blue}{\circle*{5}}}
        \put(280,190){\color{blue}{\circle*{5}}}
        \put(170,160){\color{blue}{\circle*{5}}}
        \put(200,180){\color{blue}{\circle*{5}}}
        \put(140,140){\color{blue}{\circle*{5}}}
        \put(150,190){\color{blue}{\circle*{5}}}
        \put(120,170){\color{blue}{\circle*{5}}}
        
        \put(230,70){\color{blue}{\circle*{5}}}
        \put(260,90){\color{blue}{\circle*{5}}}
        \put(200,50){\color{blue}{\circle*{5}}}
        \put(170,30){\color{blue}{\circle*{5}}}
        \put(140,10){\color{blue}{\circle*{5}}}
        \put(290,110){\color{blue}{\circle*{5}}}
       
        \put(250,40){\color{blue}{\circle*{5}}}
        \put(280,60){\color{blue}{\circle*{5}}}
        \put(220,20){\color{blue}{\circle*{5}}}
        \put(270,10){\color{blue}{\circle*{5}}}
        \put(300,30){\color{blue}{\circle*{5}}}

	\linethickness{0.4mm}
	\put(210,0){\line(0,1){200}}
	\put(110,100){\line(1,0){200}}
\end{picture}
\caption{ The set $\phi_2^{-1}\left(\bar 0\right)$  } \label{inverseimage}
\end{figure}

Since the set $\phi_k^{-1}(\bar{\ell})$ is a $k$-distance dominating set of $\ZZ^2$ and the set $Y_{m+2k,n+2k}$ is a $k$-distance neighborhood of $G_{m,n}$,
the intersection of these sets $\phi_k^{-1}\left (\bar{\ell} \right ) \cap Y_{m+2k,n+2k}$ is a $k$-distance dominating set of $G_{m,n}$ for all $\bar{\ell} \in \ZZ_p$.
By moving each vertex in the set $\phi_k^{-1}\left (\bar{\ell} \right ) \cap \left ( Y_{m+2k,n+2k} - G_{m,n} \right )$ to its nearest neighbor inside $G_{m,n}$ we obtain a dominating set $S \subset G_{m,n}$. 
Figure \ref{figure1} illustrates this construction for 3-distance domination of $G_{6,6}$ (the resulting dominating set $S$ is highlighted in red).  

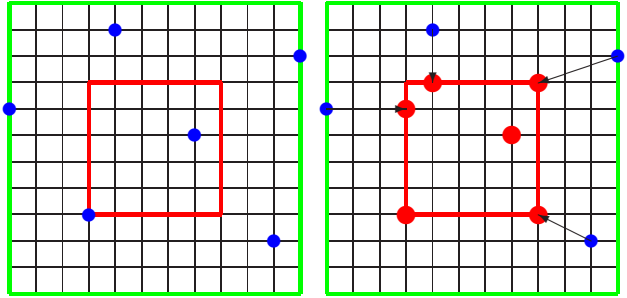
\begin{figure}[H]
\begin{picture}(410,110)(50,0)
\linethickness{.1 mm}
\multiput(130,10)(0,10){10}{\line(1,0){110}}
\multiput(140,0)(10,0){10}{\line(0,1){110}}
\linethickness{0.4mm}
    \put(160,30){\color{red}{\line(1,0){50}}}
    \put(210,30){\color{red}{\line(0,1){50}}}
    \put(160,30){\color{red}{\line(0,1){50}}}
    \put(160,80){\color{red}{\line(1,0){50}}}

    \put(130,0){\color{green}{\line(1,0){110}}}
    \put(240,0){\color{green}{\line(0,1){110}}}
    \put(130,0){\color{green}{\line(0,1){110}}}
    \put(130,110){\color{green}{\line(1,0){110}}}

\linethickness{.1 mm}
    \put(160,30){\color{blue}{\circle*{5}}}
  \put(200,60){\color{blue}{\circle*{5}}}
\put(130,70){\color{blue}{\circle*{5}}}
    \put(170,100){\color{blue}{\circle*{5}}}
\put(230,20){\color{blue}{\circle*{5}}}
\put(240,90){\color{blue}{\circle*{5}}}

\linethickness{.1 mm}
\multiput(250,10)(0,10){10}{\line(1,0){110}}
\multiput(260,0)(10,0){10}{\line(0,1){110}}
\linethickness{0.4mm}
    \put(280,30){\color{red}{\line(1,0){50}}}
    \put(330,30){\color{red}{\line(0,1){50}}}
    \put(280,30){\color{red}{\line(0,1){50}}}
    \put(280,80){\color{red}{\line(1,0){50}}}

    \put(250,0){\color{green}{\line(1,0){110}}}
    \put(360,0){\color{green}{\line(0,1){110}}}
    \put(250,0){\color{green}{\line(0,1){110}}}
    \put(250,110){\color{green}{\line(1,0){110}}}

\linethickness{.1 mm}
    \put(280,30){\color{red}{\circle*{7}}}
  \put(320,60){\color{red}{\circle*{7}}}
  
\put(250,70){\color{blue}{\circle*{5}}}
    \put(290,100){\color{blue}{\circle*{5}}}
\put(350,20){\color{blue}{\circle*{5}}}
\put(360,90){\color{blue}{\circle*{5}}}

\put(330,30){\color{red}{\circle*{7}}}
\put(290,80){\color{red}{\circle*{7}}}
\put(330,80){\color{red}{\circle*{7}}}
\put(280,70){\color{red}{\circle*{7}}}
\put(250,70){\vector(1,0){30}}
\put(290,100){\vector(0,-1){20}}
\put(360,90){\vector(-3,-1){30}}
\put(350,20){\vector(-2,1){20}}
\end{picture}
\caption{ The grid $G_{6,6}$, its neighborhood $Y_{12,12}$, and a $3$-distance dominating set.} \label{figure1}
\end{figure}

In the next section we will give an upper bound on the number of vertices in the set $S$ and show that certain vertices can be removed 
from each corner of the set $S$ and still $k$-distance dominate $G_{m,n}$.

\section{Finding an upper bound for $\left |\phi^{-1}_k\left ( \bar \ell \right ) \cap Y_{m+2k,n+2k} \right |$ } \label{section3}
Let  $p = 2k^2+2k+1$ and  $\phi_k: \ZZ^2 \rightarrow \ZZ_p$ be defined by $(i,j) \mapsto (k+1)i + kj$ as in Section \ref{section2}. 
The following lemma proves that the inverse image $\phi_k^{-1}\left (\bar \ell \right )$ contains exactly one vertex in any $p$ consecutive vertices in any row or column of $\ZZ^2$.  

\begin{lemma}\label{pconsecutive}
Let $\bar{\ell} \in \ZZ_p$.  
Then every $p$ consecutive vertices in any row or column of $G_{m,n}$ will contain exactly one element of $\phi_k^{-1}\left (\bar{\ell} \right )$.  
\end{lemma}
\begin{proof}
Recall that  $(i,j)$ is in $\phi_k^{-1}(\bar{\ell})$ for some $\bar{\ell} \in \ZZ_p$ if and only if    \[ \phi_k( (i,j) ) = (k+1)i + kj = \bar{\ell} \in \ZZ_p. \] 
Suppose now that $(i,j) \in \ZZ^2$ is in $\phi_k^{-1}(\bar{\ell})$.  

We will show that the points $(i \pm p,j)$ and $(i,j \pm p)$ are the closest points to $(i,j)$ in $\phi_k^{-1}\left (\bar{\ell} \right )$ contained 
in the same row or column as $(i,j)$. Let $a \in \ZZ$ be any integer.  In the quotient ring $\ZZ_p$ we calculate
\begin{align*} \phi_k( (i + ap,j) ) & = \overline{ (k+1)(i+ap) + kj} \\ &= \overline{[ (k+1)i+kj ] +(k+1)ap } \\ & = \overline{[ \ell ] + (k+1)ap } \\ & = \bar{\ell}  \end{align*}  and
\begin{align*} \phi_k( (i,j+ap) ) & = \overline{(k+1)(i) + k(j+ap)} \\ &= \overline{[ (k+1)i+kj ] +kap} \\ & = \overline{[\ell ] + kap} \\ & = \bar{\ell}  . \end{align*} 
Thus we see that $(i + ap,j)$ and $(i,j + ap)$ are also in $\phi_k^{-1}\left (\bar{\ell} \right )$ for any $a \in \ZZ$.  

Suppose that $0 < q <p$.  We will show that $(i \pm q, j) \notin \phi_k^{-1}\left (\bar{\ell} \right )$.  
 \begin{align*} \phi_k( (i \pm ,j) ) & = \overline{ (k+1)(i \pm q) + kj } \\ &= \overline{[ (k+1)i+kj ] \pm (k+1)q } \\ & = \overline{[ \ell ] \pm (k+1)q }  \end{align*}
 First note that that $0< k+1,q <p$. Hence $(k+1)q$ is a multiple of $p$ if and only if $(k+1)q =p$.  
 Now note that $p= 2k^2+2k+1$ has no real roots.  Thus $p$ can not possibly factor as the product $(k+1)q$ for any $0 < q <p$, and therefore $\overline{[ \ell ] \pm (k+1)q } \neq \bar{\ell}$ in $\ZZ_p$.  

Similarly, we note that $p= 2k^2+2k+1$ can not factor as $kq$ for any $0< q < p$.  Hence we see that 
that $(i, j\pm q) \notin \phi_k^{-1}\left (\bar{\ell} \right )$ for any $0 < q <p$ by computing
\begin{align*} \phi_k( (i,j \pm q ) ) & = \overline{ (k+1)i + k(j \pm q)} \\ &= \overline{[ (k+1)i+kj ] \pm kq } \\ & = \overline{[ \ell ] \pm kq } \neq \bar{\ell} .\end{align*}
This completes our proof that the points $(i \pm p,j)$ and $(i,j \pm p)$ are the closest points to $(i,j)$ in $\phi_k^{-1}\left (\bar{\ell} \right )$ contained 
in the same row or column as $(i,j)$, and thus we conclude that every $p$ consecutive vertices in any row or column $G_{m,n}$ will contain exactly one element from the set $\phi^{-1}_k(\ell)$.  
\end{proof}

Our next result uses Lemma \ref{pconsecutive}  to count the cardinality of the set $\phi_k^{-1}\left (\bar{\ell} \right ) \cap G_{m,n}$ for any $\bar \ell \in \ZZ_p$
when either $m$ or $n$ is a multiple of $p$. 

\begin{lemma}\label{multipleofp}
 If either $m$ or $n$ is a multiple of $p$, then 
 for any $\bar{\ell} \in \ZZ_p$ the cardinality of the set $\phi_k^{-1}\left(\bar{\ell} \right ) \cap G_{m,n}$ is 
 \[ | \phi_k^{-1}\left (\bar{\ell} \right ) \cap G_{m,n}|  = \frac{mn}{p}. \]
\end{lemma}

\begin{proof}
By Lemma \ref{pconsecutive}, we know for every $\bar{\ell} \in \ZZ_p$ that
every $p$ consecutive vertices in any row or column of $G_{m,n}$
will contain exactly one element of $\phi_k^{-1}\left (\bar{\ell} \right )$.  
If $m = ap$ then every row of $G_{m,n}$ will have exactly $a$ vertices from $\phi_k^{-1} \left ( \bar \ell \right )$ in it. 
Similarly, if $n= bp$ then every column of $G_{m,n}$ has $b$ vertices from $\phi_k^{-1} \left ( \bar \ell \right )$ in it. 
 Hence $ | \phi_k^{-1} \left (\bar{\ell} \right ) \cap G_{m,n}|  = \frac{mn}{p} $ if either $m$ or $n$ is a multiple of $p$. 
\end{proof}

When neither $m$ nor $n$ is a multiple of $p$, it is considerably harder to count the elements in the set 
$\phi_k^{-1}\left (\bar{\ell} \right ) \cap G_{m,n}$ for a particular $\bar \ell \in \ZZ_p$.  
However, our next result proves that there is at least one $\bar \ell \in \ZZ_p$ for which the cardinality of this set is bounded above by $\left \lfloor \frac{mn}{p} \right \rfloor$. 

\begin{proposition} \label{bound}
 If neither $m$ nor $n$ is a multiple of $p$, then there exists an $\bar{\ell} \in \ZZ_p$ such that the cardinality of the set $\phi_k^{-1}\left (\bar{\ell} \right ) \cap G_{m,n}$ satisfies 
 \[ | \phi_k^{-1}\left (\bar{\ell} \right ) \cap G_{m,n} | \leq \left \lfloor \frac{mn}{p} \right \rfloor. \]
\end{proposition}
\begin{proof}
To prove our claim, we will suppose that for some $1 \leq n \leq m < p$ and for all $\bar \ell \in \ZZ_p$ 
that $| \phi_k^{-1}\left (\bar{\ell} \right ) \cap G_{m,n} | > \left \lfloor \frac{mn}{p} \right \rfloor$ and derive a contradiction.
Note that this is equivalent to  assuming that \begin{equation} | \phi_k^{-1}\left (\bar{\ell} \right ) \cap G_{m,n} | \geq \left \lfloor \frac{mn}{p} \right \rfloor + 1 
\label{absurd} \end{equation}  for all $\bar{\ell} \in \ZZ_p$.
 
Now we consider the $mp$ by $np$ grid $G_{mp,np}$.  By Lemma \ref{multipleofp} we know that for any $\bar{\ell} \in \ZZ_p$ we have
$|\phi_k^{-1}\left (\bar{\ell} \right ) \cap G_{mp,np}| = mnp$.  We can also partition $G_{mp,np}$ into $p^2$ many copies of $G_{m,n}$.  
Supposing that Equation (\ref{absurd}) is true for all $\bar{\ell} \in \ZZ_p$, we derive the following absurdity
\begin{align*}  \left | \phi_k^{-1}\left (\bar{\ell} \right ) \cap G_{mp,np} \right |  &  \geq p^2 \left ( 
\left \lfloor \frac{mn}{p} \right \rfloor + 1 \right)\\ & = \left \lfloor mnp \right \rfloor + p^2 \\ &  = mnp+p^2\\ & > mnp = \left | \phi_k^{-1}\left (\bar{\ell} \right ) \cap G_{mp,np} \right |. \end{align*} 
This proves that Equation (\ref{absurd}) cannot be true for every $\bar{\ell} \in \ZZ_p$.  Hence we conclude that there exists an $\bar{\ell} \in \ZZ_p$ such that the cardinality of the set $\phi_k^{-1}\left (\bar{\ell} \right ) \cap G_{m,n}$ satisfies 
 $ | \phi_k^{-1}\left (\bar{\ell} \right ) \cap G_{m,n} | \leq \left \lfloor \frac{mn}{p} \right \rfloor $ as desired.
\end{proof}

\begin{corollary} \label{boundY}
  For any $m$ and $n$ there exists an $\bar{\ell} \in \ZZ_p$ such that the cardinality of the set $\phi_k^{-1}\left (\bar{\ell} \right ) \cap Y_{m+2k,n+2k}$ satisfies 
\[ \left | \phi_k^{-1}\left (\bar{\ell} \right ) \cap Y_{m+2k,n+2k} \right  | \leq \left \lfloor \frac{(m+2k)(n+2k)}{p} \right \rfloor. \]
\end{corollary}

\begin{proof}
Note that the neighborhood $Y_{m+2k,n+2k}$ is isomorphic to the grid $G_{m+2k,n+2k}$ by its definition.  Hence we can apply Lemma \ref{multipleofp} to deduce that 
\[ \left | \phi_k^{-1}\left (\bar{\ell} \right ) \cap Y_{m+2k,n+2k} \right  | =  \frac{(m+2k)(n+2k)}{p} =  \left \lfloor \frac{(m+2k)(n+2k)}{p} \right \rfloor \]
for all $\bar \ell \in \ZZ_p$ when either $m+2k$ or $n+2k$ is a multiple of $p$.
When neither $m+2k$ nor $n+2k$ is a multiple of $p$ we can apply
Proposition \ref{bound}  to conclude that there exists an $\bar \ell \in \ZZ_p$ such that $ \left | \phi_k^{-1}\left (\bar{\ell} \right ) \cap Y_{m+2k,n+2k} \right  | \leq \left \lfloor \frac{(m+2k)(n+2k)}{p} \right \rfloor$
otherwise.  
\end{proof}

Note that since $\phi_k^{-1} \left ( \bar \ell \right )  \cap Y_{m+2k,n+2k}$ is a $k$-distance dominating set for $G_{m,n}$, Corollary \ref{boundY} 
proves that $\gamma_k(G_{m,n}) \leq \left \lfloor \frac{(m+2k)(n+2k)}{p} \right \rfloor$.

\section{Main Result} \label{section4}
In the last section, we proved that $\gamma_k(G_{m,n}) \leq \left \lfloor \frac{(m+2k)(n+2k)}{p} \right \rfloor$. This bound already improves on any previously known result!
In this section, we describe three techniques which allow us to remove at least one vertex from each corner of 
$\phi_k^{-1} \left ( \bar \ell \right )  \cap Y_{m+2k,n+2k}$
to obtain a set that still dominates $G_{m,n}$. As a result, we prove that $\gamma_k(G_{m,n}) \leq \left \lfloor \frac{(m+2k)(n+2k)}{p} \right \rfloor - 4$.

\begin{lemma} \label{remove}
Suppose that $m$ and $n$ are both greater than $2p$.  
Then an element can be removed from each corner of $\phi_k^{-1}(\bar{\ell}) \cap Y_{m+2k,n+2k}$ and the resulting set still dominates $G_{m,n}$.
\end{lemma}

\begin{proof}
We will now describe how to remove at least one vertex from the northwest corner of $\phi_k^{-1}(\bar{\ell}) \cap Y_{m+2k,n+2k}$.   
For a fixed $\bar \ell \in \ZZ_p$, the other three corners of the dominating set $\phi^{-1}_k \left( \bar \ell \right) \cap Y_{m+2k,n+2k}$ are 
all either rotations or mirror images of the northwest corner
of $\phi^{-1}_k \left( \bar \ell' \right) \cap Y_{m+2k,n+2k}$ for some $\bar \ell' \in \ZZ_p$. Hence they are all isomorphic to one of the cases considered below, and thus we 
can remove a vertex from each of them as well.  
(We assume that $m$ and $n$ are both greater than $2p$ so that we can remove one vertex from each corner, and none of the local shifts effect the other three corners.)

We start by introducing the following notation: 
 We let the westernmost element in $\phi_k^{-1}(\bar{\ell}) \cap Y_{m+2k,n+2k}$ on the northern boundary of $Y_{m+2k,n+2k}$ be denoted $s$. 
 We let the northernmost element in $\phi_k^{-1}(\bar{\ell}) \cap Y_{m+2k,n+2k}$ that is one column to the west of the western boundary of $G_{m,n}$ be called $z$. 
Finally, we label the line through $s$ and $z$ by $L_1$ and the line through $s$ with slope $k/(k+1)$ by $L_2$.

Our techniques for removing a vertex from the northwest corner of $\phi_k^{-1}(\bar{\ell}) \cap Y_{m+2k,n+2k}$ depend on the slopes of $L_1$ and $L_2$, and they break down into three cases:
Either the slope of $L_1$ is negative, the slope of $L_1$ is greater than the slope of $L_2$, or the slope of $L_1$ is positive but less than or equal to the slope $L_2$.

\begin{figure}[H]
\centering
\begin{minipage}[t]{.45\textwidth}
\centering

\begin{picture}(210,230)(80,-10)
\linethickness{0.4mm}
         \put(110,200){\color{green}{\line(1,0){200}}}
         \put(110,0){\color{green}{\line(0,1){200}}}
         
         \put(140,170){\color{red}{\line(1,0){170}}}
         \put(140,0){\color{red}{\line(0,1){170}}}
         
         \put(110,200){\color{black}{\line(1,-6){18}}}

\linethickness{.1 mm}
	\multiput(110,10)(0,10){19}{\line(1,0){200}}
	\multiput(120,0)(10,0){19}{\line(0,1){200}}

        \put(120,20){\color{blue}{\circle*{5}}}

        \put(130,90){\color{blue}{\circle*{5}}}
        \put(160,50){\color{blue}{\circle*{5}}}
        \put(190,10){\color{blue}{\circle*{5}}}
        
        \put(110,200){\color{red}{\circle*{5}}}
        \put(140,160){\color{blue}{\circle*{5}}}
        \put(170,120){\color{blue}{\circle*{5}}}
        \put(200,80){\color{blue}{\circle*{5}}}
        \put(230,40){\color{blue}{\circle*{5}}}
        \put(260,0){\color{blue}{\circle*{5}}}
        
        \put(180,190){\color{blue}{\circle*{5}}}
        \put(210,150){\color{blue}{\circle*{5}}}
        \put(240,110){\color{blue}{\circle*{5}}}
        \put(270,70){\color{blue}{\circle*{5}}}
        \put(300,30){\color{blue}{\circle*{5}}}

        \put(250,180){\color{blue}{\circle*{5}}}
        \put(280,140){\color{blue}{\circle*{5}}}
        \put(310,100){\color{blue}{\circle*{5}}}
        
        \put(119,145){$L_1$}
        \put(115,200){$s$}
        \put(110,170){\color{red}{\line(1,1){30}}}
         \put(110,170){\color{red}{\line(-1,1){30}}}
         \put(110,230){\color{red}{\line(-1,-1){30}}}
         \put(110,230){\color{red}{\line(1,-1){30}}}
	    \put(120,90){$z$}
        \put(130,60){\color{red}{\line(1,1){30}}}
        \put(130,60){\color{red}{\line(-1,1){30}}}
        \put(130,120){\color{red}{\line(-1,-1){30}}}
        \put(130,120){\color{red}{\line(1,-1){30}}}
\end{picture}

\caption{Case 1 before shifts} \label{case11}
\end{minipage}\hfill
\begin{minipage}[t]{.45\textwidth}
\centering

\begin{picture}(210,230)(100,-10)
\linethickness{0.4mm}
         \put(110,200){\color{green}{\line(1,0){200}}}
         \put(110,0){\color{green}{\line(0,1){200}}}
         \put(140,170){\color{red}{\line(1,0){170}}}
         \put(140,0){\color{red}{\line(0,1){170}}}

\linethickness{.1 mm}
	\multiput(110,10)(0,10){19}{\line(1,0){200}}
	\multiput(120,0)(10,0){19}{\line(0,1){200}}

        \put(140,20){\color{red}{\circle*{5}}}

        \put(140,90){\color{red}{\circle*{5}}}
        \put(160,50){\color{blue}{\circle*{5}}}
        \put(190,10){\color{blue}{\circle*{5}}}
        
        \put(140,160){\color{blue}{\circle*{5}}}
        \put(170,120){\color{blue}{\circle*{5}}}
        \put(200,80){\color{blue}{\circle*{5}}}
        \put(230,40){\color{blue}{\circle*{5}}}
        \put(260,0){\color{blue}{\circle*{5}}}
        
        \put(180,170){\color{red}{\circle*{5}}}
        \put(210,150){\color{blue}{\circle*{5}}}
        \put(240,110){\color{blue}{\circle*{5}}}
        \put(270,70){\color{blue}{\circle*{5}}}
        \put(300,30){\color{blue}{\circle*{5}}}

        \put(250,170){\color{red}{\circle*{5}}}
        \put(280,140){\color{blue}{\circle*{5}}}
        \put(310,100){\color{blue}{\circle*{5}}}        
\end{picture}

\caption{Case 1 after the shifts} \label{case12}
\end{minipage}
\end{figure}
\textbf{Case 1:}
If the slope of $L_1$ is negative as depicted in Figures \ref{case11} and \ref{case12}, then the $k$-distance neighborhood of $s$ does not intersect $G_{m,n}$. 
Hence, $s$ can be removed from $\phi_k^{-1}(\bar{\ell}) \cap Y_{m+2k,n+2k}$ and the resulting set still dominates $G_{m,n}$.  
To obtain a dominating set of $G_{m,n}$ that is contained entirely in $G_{m,n}$, move each element of $\phi_k^{-1}(\bar{\ell}) \cap (Y_{m+2k,n+2k} - G_{m,n})$ to its nearest neighbor 
in $G_{m,n}$.

\textbf{Case 2:} 
 If the slope of $L_1$ is greater than the slope of $L_2$, then  shift all of the elements northwest of $L_1$ to the east one unit so that we can remove $s$.  
As depicted in Figure \ref{case21}, let the southernmost vertex in the $k$-distance neighborhood of $s$ be denoted $u$.  (It lies on the northern boundary of $G_{m,n}$ and is due south of $s$.)
Let the vertex at the intersection of the northern boundary of $G_{m,n}$ and $L_2$ be denoted $t$.  (It lies $k+1$ vertices to the west of $u$.)

Note that after shifting all of the elements northwest of $L_1$ to the east one unit, the $k$ distance neighborhood of $t$ will contain $u$.  
Hence $s$ can be removed from our dominating set. 
The previous shift leaves the vertex $b$ on the western boundary of $G_{m,n}$ undominated.  Note that the vertex $b$ is $k+1$ vertices north of $z$,
so we can shift the vertex $z$ up one unit, and the $k$-distance neighborhood of $z$ will contain $b$ and all of the vertices 
that $z$ originally dominated before these two shifts. (The original domination neighborhood of $z$ is highlighted by circles in Figure \ref{case22}.)  
Finally, we move every vertex in this dominating set that lies outside $G_{m,n}$ to its nearest neighbor inside $G_{m,n}$ to obtain a dominating set that is contained inside of $G_{m,n}$.  
\begin{figure}[H]
\centering
\begin{minipage}[t]{.45\textwidth}
\centering

\begin{picture}(210,240)(100,-25)

\linethickness{0.4mm}
         \put(110,200){\color{green}{\line(1,0){230}}}
         \put(110,-30){\color{green}{\line(0,1){230}}}
         
         \put(230,200){\color{blue}{\line(-4,-3){120}}}
         \put(230,200){\color{black}{\line(-1,-2){100}}}
         
         \put(140,170){\color{red}{\line(1,0){200}}}
         \put(140,-30){\color{red}{\line(0,1){200}}}

\linethickness{.1 mm}
	\multiput(110,-20)(0,10){22}{\line(1,0){230}}
	\multiput(120,-30)(10,0){22}{\line(0,1){230}}
	
        \put(140,70){\color{blue}{\circle*{5}}}
        \put(140,40){\color{red}{\line(1,1){30}}}
        \put(140,40){\color{red}{\line(-1,1){30}}}
        \put(140,100){\color{red}{\line(-1,-1){30}}}
        \put(140,100){\color{red}{\line(1,-1){30}}}
        
        \put(180,100){\color{blue}{\circle*{5}}}
	    \put(180,70){\color{red}{\line(1,1){30}}}
        \put(180,70){\color{red}{\line(-1,1){30}}}
        \put(180,130){\color{red}{\line(-1,-1){30}}}
        \put(180,130){\color{red}{\line(1,-1){30}}}
        
        \put(220,130){\color{blue}{\circle*{5}}}
        \put(260,160){\color{blue}{\circle*{5}}}
        \put(300,190){\color{blue}{\circle*{5}}}
        
        \put(140,20){\color{red}{\circle{5}}}
        \put(140,10){\color{red}{\circle{5}}}
        \put(140,0){\color{red}{\circle{5}}}
        \put(140,-10){\color{red}{\circle{5}}}
        \put(140,-20){\color{red}{\circle{5}}}
        \put(150,10){\color{red}{\circle{5}}}
        \put(150,0){\color{red}{\circle{5}}}
        \put(150,-10){\color{red}{\circle{5}}}
        \put(160,0){\color{red}{\circle{5}}}

        \put(110,110){\color{blue}{\circle*{5}}}
        \put(110,80){\color{red}{\line(1,1){30}}}
        \put(110,80){\color{red}{\line(-1,1){30}}}
        \put(110,140){\color{red}{\line(-1,-1){30}}}
        \put(110,140){\color{red}{\line(1,-1){30}}}

        \put(150,140){\color{blue}{\circle*{5}}}
        \put(150,110){\color{red}{\line(1,1){30}}}
        \put(150,110){\color{red}{\line(-1,1){30}}}
        \put(150,170){\color{red}{\line(-1,-1){30}}}
        \put(150,170){\color{red}{\line(1,-1){30}}}
        
        \put(190,170){\color{blue}{\circle*{5}}}
        \put(230,200){\color{red}{\circle*{5}}}
        \put(183,171){$t$}

        \put(130,0){\color{blue}{\circle*{5}}}
        \put(170,30){\color{blue}{\circle*{5}}}
        \put(210,60){\color{blue}{\circle*{5}}}
        \put(250,90){\color{blue}{\circle*{5}}}
        \put(290,120){\color{blue}{\circle*{5}}}

        \put(120,180){\color{blue}{\circle*{5}}}
	    \put(120,150){\color{red}{\line(1,1){30}}}
        \put(120,150){\color{red}{\line(-1,1){30}}}
        \put(120,210){\color{red}{\line(-1,-1){30}}}
        \put(120,210){\color{red}{\line(1,-1){30}}}
        
        \put(330,150){\color{blue}{\circle*{5}}}
        \put(320,80){\color{blue}{\circle*{5}}}
        
        \put(240,20){\color{blue}{\circle*{5}}}
        \put(280,50){\color{blue}{\circle*{5}}}
        \put(310,10){\color{blue}{\circle*{5}}}
        
        \put(110,110){\vector(1,0){10}}
    	\put(150,140){\vector(1,0){10}}
	    \put(190,170){\vector(1,0){10}}
	    \put(190,140){\color{red}{\line(1,1){30}}}
        \put(190,140){\color{red}{\line(-1,1){30}}}
        \put(190,200){\color{red}{\line(-1,-1){30}}}
        \put(190,200){\color{red}{\line(1,-1){30}}}
	    
	    \put(200,-10){\color{blue}{\circle*{5}}}
	    \put(270,-20){\color{blue}{\circle*{5}}}
	    
	    \put(180,100){\vector(1,0){10}}
    	\put(120,180){\vector(1,0){10}}
    	\put(130,0){\vector(1,0){10}}
    	\put(140,0){\vector(0,1){10}}
    	\put(140,70){\vector(1,0){10}}
    	

	    \put(140,40){\color{black}{\circle*{5}}}
	    \put(133,41){$b$}
	    
	    \put(121,1){$z$}
	    \put(130,-30){\color{red}{\line(1,1){30}}}
        \put(130,-30){\color{red}{\line(-1,1){30}}}
        \put(130,30){\color{red}{\line(-1,-1){30}}}
        \put(130,30){\color{red}{\line(1,-1){30}}}
        
        \put(201,134){$L_1$}
        \put(104,117){$L_2$}
        
	    \put(233,201){$s$}
	    \put(230,170){\color{red}{\line(1,1){30}}}
        \put(230,170){\color{red}{\line(-1,1){30}}}
        \put(230,230){\color{red}{\line(-1,-1){30}}}
        \put(230,230){\color{red}{\line(1,-1){30}}}
        
	    \put(233,171){$u$}
        \put(230,170){\color{black}{\circle*{5}}}
        
\end{picture}

\caption{Case 2 before shifts} \label{case21}
\end{minipage}\hfill
\begin{minipage}[t]{.45\textwidth}
\centering

\begin{picture}(210,240)(100,-25)

\linethickness{0.4mm}
        \put(110,200){\color{green}{\line(1,0){230}}}
        \put(110,-30){\color{green}{\line(0,1){230}}}

         \put(140,170){\color{red}{\line(1,0){200}}}
         \put(140,-30){\color{red}{\line(0,1){200}}}

\linethickness{.1 mm}
	\multiput(110,-20)(0,10){22}{\line(1,0){230}}
	\multiput(120,-30)(10,0){22}{\line(0,1){230}}
	
        \put(150,70){\color{blue}{\circle*{5}}}
        \put(150,40){\color{red}{\line(1,1){30}}}
        \put(150,40){\color{red}{\line(-1,1){30}}}
        \put(150,100){\color{red}{\line(-1,-1){30}}}
        \put(150,100){\color{red}{\line(1,-1){30}}}
        
        \put(190,100){\color{blue}{\circle*{5}}}
         \put(190,70){\color{red}{\line(1,1){30}}}
        \put(190,70){\color{red}{\line(-1,1){30}}}
        \put(190,130){\color{red}{\line(-1,-1){30}}}
        \put(190,130){\color{red}{\line(1,-1){30}}}
        
	    \put(233,171){$u$}
        \put(230,170){\color{black}{\circle*{5}}}	    
        
        \put(220,130){\color{blue}{\circle*{5}}}
        \put(260,160){\color{blue}{\circle*{5}}}
    
    	\put(140,40){\color{black}{\circle*{5}}}
	    \put(133,41){$b$}
        
        \put(193,171){$t$}

        \put(140,80){\color{red}{\line(1,1){30}}}
        \put(140,80){\color{red}{\line(-1,1){30}}}
        \put(140,140){\color{red}{\line(-1,-1){30}}}
        \put(140,140){\color{red}{\line(1,-1){30}}}

	    \put(131,11){$z$}

        \put(170,30){\color{blue}{\circle*{5}}}
        \put(210,60){\color{blue}{\circle*{5}}}
        \put(250,90){\color{blue}{\circle*{5}}}
        \put(290,120){\color{blue}{\circle*{5}}}

        \put(330,150){\color{blue}{\circle*{5}}}
        \put(320,80){\color{blue}{\circle*{5}}}
        
        \put(200,-10){\color{blue}{\circle*{5}}}
	    \put(270,-20){\color{blue}{\circle*{5}}}
        
        \put(240,20){\color{blue}{\circle*{5}}}
        \put(280,50){\color{blue}{\circle*{5}}}
        \put(310,10){\color{blue}{\circle*{5}}}

        \put(220,130){\color{blue}{\circle*{5}}}
        \put(260,160){\color{blue}{\circle*{5}}}
        
        \put(300,170){\color{red}{\circle*{5}}}

        \put(140,110){\color{red}{\circle*{5}}}
        
        \put(160,140){\color{blue}{\circle*{5}}}
        \put(160,110){\color{red}{\line(1,1){30}}}
        \put(160,110){\color{red}{\line(-1,1){30}}}
        \put(160,170){\color{red}{\line(-1,-1){30}}}
        \put(160,170){\color{red}{\line(1,-1){30}}}
        
        \put(200,170){\color{blue}{\circle*{5}}}
        \put(200,140){\color{red}{\line(1,1){30}}}
        \put(200,140){\color{red}{\line(-1,1){30}}}
        \put(200,200){\color{red}{\line(-1,-1){30}}}
        \put(200,200){\color{red}{\line(1,-1){30}}}
        
        \put(140,10){\color{red}{\circle*{5}}}
        \put(140,-20){\color{red}{\line(1,1){30}}}
        \put(140,-20){\color{red}{\line(-1,1){30}}}
        \put(140,40){\color{red}{\line(-1,-1){30}}}
        \put(140,40){\color{red}{\line(1,-1){30}}}
        
        \put(140,20){\color{red}{\circle{5}}}
        \put(140,10){\color{red}{\circle{5}}}
        \put(140,0){\color{red}{\circle{5}}}
        \put(140,-10){\color{red}{\circle{5}}}
        \put(140,-20){\color{red}{\circle{5}}}
        \put(150,10){\color{red}{\circle{5}}}
        \put(150,0){\color{red}{\circle{5}}}
        \put(150,-10){\color{red}{\circle{5}}}
        \put(160,0){\color{red}{\circle{5}}}
        
        \put(140,170){\color{red}{\circle*{5}}}
        \put(140,140){\color{red}{\line(1,1){30}}}
        \put(140,140){\color{red}{\line(-1,1){30}}}
        \put(140,200){\color{red}{\line(-1,-1){30}}}
        \put(140,200){\color{red}{\line(1,-1){30}}}

\end{picture}

\caption{Case 2 after the shifts} \label{case22}
\end{minipage}
\end{figure}
\textbf{Case 3:} If the slope of $L_2$ is greater than or equal to the slope of $L_1$, then we can shift all vertices in  $\phi_k^{-1}(\bar{\ell}) \cap Y_{m+2k,n+2k}$ that lie on $L_1$ to the east one unit as shown in Figure \ref{Case31} which causes $t$ to dominate $u$.
This allows us to remove $s$ from our dominating set, but it also creates a diagonal of uncovered vertices as shown in Figure \ref{Case32}. 

\begin{figure}[H]
\centering
\begin{minipage}[t]{.45\textwidth}
\centering

\begin{picture}(210,230)(100,-10)

\linethickness{0.4mm}
         \put(110,200){\color{green}{\line(1,0){200}}}
         \put(110,0){\color{green}{\line(0,1){200}}}
        
         \put(290,200){\color{blue}{\line(-4,-3){160}}}
         
         \put(140,170){\color{red}{\line(1,0){170}}}
         \put(140,0){\color{red}{\line(0,1){170}}}

\linethickness{.1 mm}
\multiput(110,10)(0,10){19}{\line(1,0){200}}
\multiput(120,0)(10,0){19}{\line(0,1){200}}
      
      \put(240,100){\color{blue}{\circle*{5}}}
        \put(280,130){\color{blue}{\circle*{5}}}
        \put(110,190){\color{blue}{\circle*{5}}}
      
      \put(200,70){\color{blue}{\circle*{5}}}
        \put(160,40){\color{blue}{\circle*{5}}}

        \put(210,140){\color{blue}{\circle*{5}}}
        \put(250,170){\color{blue}{\circle*{5}}}
        \put(290,200){\color{red}{\circle*{5}}}
        \put(170,110){\color{blue}{\circle*{5}}}

        \put(190,0){\color{blue}{\circle*{5}}}
        \put(230,30){\color{blue}{\circle*{5}}}
        \put(270,60){\color{blue}{\circle*{5}}}
        \put(310,90){\color{blue}{\circle*{5}}}

        \put(120,10){\color{blue}{\circle*{5}}}
        
        \put(140,150){\color{blue}{\circle*{5}}}
        \put(180,180){\color{blue}{\circle*{5}}}
        \put(130,80){\color{blue}{\circle*{5}}}
        
        \put(250,170){\vector(1,0){10}}
    \put(210,140){\vector(1,0){10}}
   \put(170,110){\vector(1,0){10}}
    \put(130,80){\vector(1,0){10}}
        
   \put(244,172){$t$}
        \put(282,163){$u$}
        \put(250,140){\color{red}{\line(1,1){30}}}
        \put(250,140){\color{red}{\line(-1,1){30}}}
        \put(250,200){\color{red}{\line(-1,-1){30}}}
        \put(250,200){\color{red}{\line(1,-1){30}}}

        \put(295,200){$s$}
        \put(290,170){\color{red}{\line(1,1){30}}}
        \put(290,170){\color{red}{\line(-1,1){30}}}
        \put(290,230){\color{red}{\line(-1,-1){30}}}
        \put(290,230){\color{red}{\line(1,-1){30}}}
         
        \put(140,120){\color{red}{\line(1,1){30}}}
        
        \put(110,160){\color{red}{\line(1,1){30}}}
        \put(110,160){\color{red}{\line(-1,1){30}}}
        \put(110,220){\color{red}{\line(-1,-1){30}}}
        \put(110,220){\color{red}{\line(1,-1){30}}}
   
        \put(140,120){\color{red}{\line(-1,1){30}}}
        \put(140,180){\color{red}{\line(-1,-1){30}}}
        \put(140,180){\color{red}{\line(1,-1){30}}}
        
        \put(180,150){\color{red}{\line(1,1){30}}}
        \put(180,150){\color{red}{\line(-1,1){30}}}
        \put(180,210){\color{red}{\line(-1,-1){30}}}
        \put(180,210){\color{red}{\line(1,-1){30}}}
        
        \put(210,110){\color{red}{\line(1,1){30}}}
        \put(210,110){\color{red}{\line(-1,1){30}}}
        \put(210,170){\color{red}{\line(-1,-1){30}}}
        \put(210,170){\color{red}{\line(1,-1){30}}}
        
        \put(170,80){\color{red}{\line(1,1){30}}}
        \put(170,80){\color{red}{\line(-1,1){30}}}
        \put(170,140){\color{red}{\line(-1,-1){30}}}
        \put(170,140){\color{red}{\line(1,-1){30}}}
        
        \put(122,82){$z$}
        \put(190,130){$L_1=L_2$}
        \put(130,50){\color{red}{\line(1,1){30}}}
        \put(130,50){\color{red}{\line(-1,1){30}}}
        \put(130,110){\color{red}{\line(-1,-1){30}}}
        \put(130,110){\color{red}{\line(1,-1){30}}}
\put(290,170){\color{blue}{\circle{5}}}
\end{picture}

\caption{Case 3 before shifts} \label{Case31}
\end{minipage} \hfill
\begin{minipage}[t]{.45\textwidth}
\centering
\begin{picture}(210,230)(100,-10)
\linethickness{0.4mm}
         \put(110,200){\color{green}{\line(1,0){200}}}
         \put(110,0){\color{green}{\line(0,1){200}}}
             \put(290,200){\color{blue}{\line(-4,-3){160}}}

         \put(140,170){\color{red}{\line(1,0){170}}}
         \put(140,0){\color{red}{\line(0,1){170}}}

\linethickness{.1 mm}
\multiput(110,10)(0,10){19}{\line(1,0){200}}
\multiput(120,0)(10,0){19}{\line(0,1){200}}
      
        \put(240,100){\color{blue}{\circle*{5}}}
        \put(280,130){\color{blue}{\circle*{5}}}
        \put(110,190){\color{blue}{\circle*{5}}}
      
        \put(200,70){\color{blue}{\circle*{5}}}
        \put(160,40){\color{blue}{\circle*{5}}}

        \put(220,140){\color{blue}{\circle*{5}}}
        \put(260,170){\color{blue}{\circle*{5}}}
        \put(290,200){\color{red}{\circle*{5}}}
        \put(180,110){\color{blue}{\circle*{5}}}

        \put(190,0){\color{blue}{\circle*{5}}}
        \put(230,30){\color{blue}{\circle*{5}}}
        \put(270,60){\color{blue}{\circle*{5}}}
        \put(310,90){\color{blue}{\circle*{5}}}

        \put(120,10){\color{blue}{\circle*{5}}}
        
        \put(140,150){\color{blue}{\circle*{5}}}
        \put(180,180){\color{blue}{\circle*{5}}}
        \put(140,80){\color{blue}{\circle*{5}}}
        \put(198,138){$L_2$}
        
        \put(140,150){\vector(0,-1){10}}
        \put(180,180){\vector(0,-1){10}}    
        \put(110,190){\vector(0,-1){10}}  
        \put(210,170){\color{black}{\circle{5}}}
        \put(200,160){\color{black}{\circle{5}}}
        \put(190,150){\color{black}{\circle{5}}}
       
       \put(170,140){\color{black}{\circle{5}}}
       \put(160,130){\color{black}{\circle{5}}}
        \put(150,120){\color{black}{\circle{5}}}
        
        \put(110,160){\color{red}{\line(1,1){30}}}
        \put(110,160){\color{red}{\line(-1,1){30}}}
        \put(110,220){\color{red}{\line(-1,-1){30}}}
        \put(110,220){\color{red}{\line(1,-1){30}}}
   
        \put(254,172){$t$}
        \put(282,163){$u$}
        \put(260,140){\color{red}{\line(1,1){30}}}
        \put(260,140){\color{red}{\line(-1,1){30}}}
        \put(260,200){\color{red}{\line(-1,-1){30}}}
        \put(260,200){\color{red}{\line(1,-1){30}}}
        
        \put(295,200){$s$}
        \put(290,170){\color{red}{\line(1,1){30}}}
        \put(290,170){\color{red}{\line(-1,1){30}}}
        \put(290,230){\color{red}{\line(-1,-1){30}}}
        \put(290,230){\color{red}{\line(1,-1){30}}}
         
        \put(140,120){\color{red}{\line(-1,1){30}}}
        \put(140,120){\color{red}{\line(1,1){30}}}
        \put(140,180){\color{red}{\line(-1,-1){30}}}
        \put(140,180){\color{red}{\line(1,-1){30}}}
        
        \put(180,150){\color{red}{\line(1,1){30}}}
        \put(180,150){\color{red}{\line(-1,1){30}}}
        \put(180,210){\color{red}{\line(-1,-1){30}}}
        \put(180,210){\color{red}{\line(1,-1){30}}}
        
        \put(220,110){\color{red}{\line(1,1){30}}}
        \put(220,110){\color{red}{\line(-1,1){30}}}
        \put(220,170){\color{red}{\line(-1,-1){30}}}
        \put(220,170){\color{red}{\line(1,-1){30}}}
        
        \put(180,80){\color{red}{\line(1,1){30}}}
        \put(180,80){\color{red}{\line(-1,1){30}}}
        \put(180,140){\color{red}{\line(-1,-1){30}}}
        \put(180,140){\color{red}{\line(1,-1){30}}}

        \put(140,50){\color{red}{\line(1,1){30}}}
        \put(140,50){\color{red}{\line(-1,1){30}}}
        \put(140,110){\color{red}{\line(-1,-1){30}}}
        \put(140,110){\color{red}{\line(1,-1){30}}}
        \put(290,170){\color{blue}{\circle{5}}}
\end{picture}

\caption{Case 3 after first shift} \label{Case32}
\end{minipage}
\end{figure}

Now we take the vertices in  $\phi_k^{-1}(\bar{\ell}) \cap Y_{m+2k,n+2k}$ that are strictly northwest of $L_1$ and shift them down one unit. This shift dominates all of the vertices on the undominated diagonal. 
We then move every vertex in this dominating set that lies outside of $G_{m,n}$ to its nearest neighbor inside $G_{m,n}$  to obtain a dominating set completely contained in $G_{m,n}$.

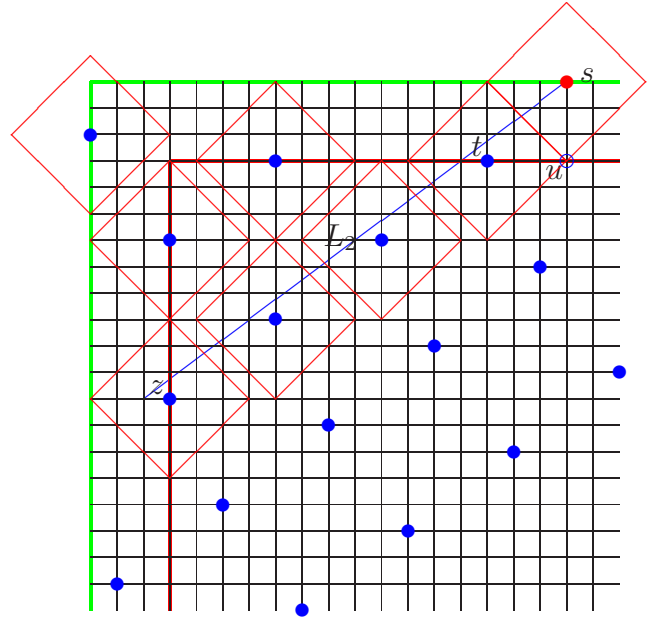
\begin{figure}[H]
\centering
\begin{picture}(210,230)(100,-10)

\linethickness{0.4mm}
         \put(110,200){\color{green}{\line(1,0){200}}}
         \put(110,0){\color{green}{\line(0,1){200}}}
             \put(290,200){\color{blue}{\line(-4,-3){160}}}
         
         \put(140,170){\color{red}{\line(1,0){170}}}
         \put(140,0){\color{red}{\line(0,1){170}}}

\linethickness{.1 mm}
\multiput(110,10)(0,10){19}{\line(1,0){200}}
\multiput(120,0)(10,0){19}{\line(0,1){200}}
      
        \put(240,100){\color{blue}{\circle*{5}}}
        \put(280,130){\color{blue}{\circle*{5}}}
        \put(110,180){\color{blue}{\circle*{5}}}
      
        \put(200,70){\color{blue}{\circle*{5}}}
        \put(160,40){\color{blue}{\circle*{5}}}

        \put(220,140){\color{blue}{\circle*{5}}}
        \put(260,170){\color{blue}{\circle*{5}}}
        \put(290,200){\color{red}{\circle*{5}}}
        \put(180,110){\color{blue}{\circle*{5}}}

        \put(190,0){\color{blue}{\circle*{5}}}
        \put(230,30){\color{blue}{\circle*{5}}}
        \put(270,60){\color{blue}{\circle*{5}}}
        \put(310,90){\color{blue}{\circle*{5}}}

        \put(120,10){\color{blue}{\circle*{5}}}
        
        \put(140,140){\color{blue}{\circle*{5}}}
        \put(180,170){\color{blue}{\circle*{5}}}
        \put(140,80){\color{blue}{\circle*{5}}}
        \put(198,138){$L_2$}

        \put(110,150){\color{red}{\line(1,1){30}}}
        \put(110,150){\color{red}{\line(-1,1){30}}}
        \put(110,210){\color{red}{\line(-1,-1){30}}}
        \put(110,210){\color{red}{\line(1,-1){30}}}
   
        \put(254,172){$t$}
        \put(282,163){$u$}
        \put(260,140){\color{red}{\line(1,1){30}}}
        \put(260,140){\color{red}{\line(-1,1){30}}}
        \put(260,200){\color{red}{\line(-1,-1){30}}}
        \put(260,200){\color{red}{\line(1,-1){30}}}
        
        \put(295,200){$s$}
        \put(290,170){\color{red}{\line(1,1){30}}}
        \put(290,170){\color{red}{\line(-1,1){30}}}
        \put(290,230){\color{red}{\line(-1,-1){30}}}
        \put(290,230){\color{red}{\line(1,-1){30}}}
         
        \put(140,110){\color{red}{\line(-1,1){30}}}
        \put(140,110){\color{red}{\line(1,1){30}}}
        \put(140,170){\color{red}{\line(-1,-1){30}}}
        \put(140,170){\color{red}{\line(1,-1){30}}}
        
        \put(180,140){\color{red}{\line(1,1){30}}}
        \put(180,140){\color{red}{\line(-1,1){30}}}
        \put(180,200){\color{red}{\line(-1,-1){30}}}
        \put(180,200){\color{red}{\line(1,-1){30}}}
        
        \put(220,110){\color{red}{\line(1,1){30}}}
        \put(220,110){\color{red}{\line(-1,1){30}}}
        \put(220,170){\color{red}{\line(-1,-1){30}}}
        \put(220,170){\color{red}{\line(1,-1){30}}}
        
        \put(180,80){\color{red}{\line(1,1){30}}}
        \put(180,80){\color{red}{\line(-1,1){30}}}
        \put(180,140){\color{red}{\line(-1,-1){30}}}
        \put(180,140){\color{red}{\line(1,-1){30}}}

        \put(132,82){$z$}
        \put(140,50){\color{red}{\line(1,1){30}}}
        \put(140,50){\color{red}{\line(-1,1){30}}}
        \put(140,110){\color{red}{\line(-1,-1){30}}}
        \put(140,110){\color{red}{\line(1,-1){30}}}
        \put(290,170){\color{blue}{\circle{5}}}
\end{picture}
\caption{Case 3 after second shift} \label{Case33}
\end{figure}

In Cases 1, 2, and 3, we have shown how to remove at least one vertex from the northwest corner of 
the dominating set $\phi^{-1}_k \left( \bar \ell \right) \cap Y_{m+2k,n+2k}$ for any $\bar \ell \in \ZZ_p$, and the other 
four corners look the same up to isomorphism. This proves that we can remove at least four vertices from
$\phi^{-1}_k \left( \bar \ell \right) \cap Y_{m+2k,n+2k}$ provided the grid $G_{m,n}$ is large enough so that the corners do not overlap.
\end{proof}

Note that the example illustrated in Figures \ref{Case31}-\ref{Case33} shows that it is sometimes possible to remove two vertices from a corner of $\phi^{-1}_k \left( \bar \ell \right) \cap Y_{m+2k,n+2k}$ when 
the slope of $L_1$ is greater than or equal to that of $L_2$, because the vertex in the northwest corner of Figure \ref{Case33} can also be removed from the dominating set. 
So there are instances where we can remove five vertices from $\phi^{-1}_k \left( \bar \ell \right) \cap Y_{m+2k,n+2k}$ and still dominate $G_{m,n}$, but that is not the case in general.  

We are now ready to prove our main result.  
\begin{theorem} \label{Maintheorem}
Assume that $m$ and $n$ are both greater than $2p$ where $p = 2k^2+2k+1$.  Then the $k$-distance domination number of an $m \times n$ grid graph $G_{m,n}$ is bounded above by 
\[ \gamma_{k} (G_{m,n} ) \leq  \left \lfloor  \frac{ (m+2k)(n+2k) }{ p} \right \rfloor -4  .\]
\end{theorem}

\begin{proof}
Corollary \ref{boundY} shows that for some $\bar \ell \in \ZZ_p$ the set $\phi^{-1}_k \left( \bar \ell \right) \cap Y_{m+2k,n+2k}$ contains at most $ \left \lfloor  \frac{ (m+2k)(n+2k) }{ p} \right \rfloor$ vertices.
Lemma \ref{remove} shows that if $m$ and $n$ are both greater than $2p$ then we can remove at least 4 vertices from the set  $\phi^{-1}_k \left( \bar \ell \right) \cap Y_{m+2k,n+2k}$  and still dominate $G_{m,n}$.
Thus we have shown $\gamma_{k} (G_{m,n} ) \leq  \left \lfloor  \frac{ (m+2k)(n+2k) }{ p} \right \rfloor -4 $.  
\end{proof}

\section{Acknowledgements} 
We thank our adviser, Erik Insko, for his hours of support on this project.  We thank Dr. Katie Johnson for many helpful comments on an earlier draft of this paper.
Finally, we thank USTARS 2014 for giving us the opportunity us to present this paper.


\begin{thebibliography}{99}
\bibitem{Ala11} S.~Alanko, S.~Crevals, A.~Isopoussu, P.~\"Ostergard, and V.~Petterson. \newblock Computing the domination number of grid graphs. \newblock {\em Electr. J. Comb.} \textbf{18}(1), 2011.


\bibitem{BIJM2014} D.~C.~Blessing, E.~Insko, K.~Johnson, and C.~Mauretour. \newblock{ \em On (t,r) broadcast domination of grids.}  \newblock{\tt{arXiv:1401.2499v1} } (preprint 2014).


\bibitem{Cha92} T.~Y.~Chang. \newblock{ \em Domination numbers of grid graphs}.  Ph.D. thesis, Dept. of Mathematics, University of South Florida, 1992.

\bibitem{ChaCla93}T.~Y.~Chang and W.~E.~Clark.   \newblock The domination numbers of the $5 \times n$ and $6 \times n$ grid graphs.  \newblock {\em J. Graph
    Theory}, \textbf{17}:81--107, 1993.

\bibitem{ChaClaHar94}T.~Y.~Chang, W.~E.~Clark, and E.~O.~Hare.  \newblock
Dominations of complete grid graphs I.  \newblock {\em Ars Combin},  \textbf{38}: 97-111, 1994.




\bibitem{CocHarHedWim85}
  \newblock{E.~J.~Cockayne, E.~O.~Hare, S.~T.~Hedetniemi, T.~V.~Wimer.} 
  \newblock{Bounds for the Domination Number of Grid Graphs},
  \newblock{\em Congressus Numeratium} {\bf 47}:217-228, 1985.

\bibitem{FatSmiSun13}
\newblock{E.~Fata, S.~L.~Smith and S.~Sundaram} 
\newblock Distributed Dominating Sets on Grids. 
\newblock { \emph{Proceedings of ACC 2013, the 32nd American Control Conference,}} 2013 (to appear).



\bibitem{GonPinRaoTho11} \newblock{D.~Gon\c{c}alves, A.~Pinlou, M.~Rao, and S.~Thomass\'e } 
 \newblock The domination number of grids.
 \newblock{ \emph{ SIAM J. Discrete Math.}}, \textbf{25}: {1443--1453}, 2011.
 

\bibitem{HarHarHed86}
  \newblock{E.~O.~Hare, W.~R.~Hare, and S.~T.~Hedetniemi}, 
  \newblock{Algorithms for Computing the Domination Number of $K\times N$ Complete Grid Graphs},
  \newblock{\em Congressus Numeratium} {\bf 55}:81--92, 1986.


\bibitem{HarHay00} \newblock{F.~Harary~and~T.W.~Haynes}, \newblock{ Double domination in graphs.} \newblock{ \em Ars Combin}. \textbf{55}:
201--213, 2000

\bibitem{HayHedSla98} \newblock{T.W.~Haynes, S.T.~Hedetniemi, and P.J.~Slater}, \newblock{Fundamentals of Domination
in Graphs,} \newblock Marcel Dekker, New York, 1998.

\bibitem{HayHedSla98b} \newblock{T.W.~Haynes,~S.T.~Hedetniemi~and~P.J.~Slater}, Domination in Graphs: Advanced Topics, \newblock Marcel Dekker, New York, 1998.



\bibitem{Hen98}
\newblock M.~Henning,
\newblock Distance Domination in Graphs.
\newblock \emph{ Domination in Graphs}. 321-349, 1998.

\bibitem{JacKin84} M.~S.~Jacobson and L.~F.~Kinch.  \newblock On the domination number of products of graphs.  \newblock {\em Ars. Combin}. \textbf{18}:43--44, 1984.

\bibitem{JotPusSugSwa2011} G.~Jothilakshmi, A.~P.~Pushpalatha, S.~Suganthi, and V.~Swaminathan.
  \newblock $(k,r)$-Domination in Graphs  \newblock {\em Int. J. Contemp. Math. Sciences}, Vol. 6,  \textbf{29}, 1439 - 1446, 2011.

%
%


\bibitem{Sla76} P.J.~Slater, \newblock { $R$-domination in graphs}, \newblock {\em J. Assoc. Comput. Mach.} \textbf{ 23} 446–450, 1976.


\bibitem{Slo94} N.~J.~A.~Sloane.  \newblock An on-line version of the encyclopedia of integer sequences. \newblock {\em Electron. J. Combin.} \textbf{1}, 1994.

\bibitem{Spa98} A.~Spalding.  \newblock {\em Min-Plus Algebra and Graph Domination}. \newblock Ph.D. thesis.  Dept. of Applied Mathematics, University of Colorado, 1998.



















\end{thebibliography}
\end{document}